\documentclass[12pt]{article}
\usepackage{graphicx}
\usepackage{amssymb}
\usepackage{amsmath}
\usepackage{amsthm}
\usepackage{color}
\usepackage{cite}

\makeatletter
\def\@biblabel#1{#1.}
\makeatother

\newcommand{\N}{{\mathbb N}}

\newcommand{\R}{{\mathbb R}}

\newcommand{\cyr}{%
  \renewcommand\rmdefault{wncyr}%
  \renewcommand\sfdefault{wncyss}%
  \renewcommand\encodingdefault{OT2}%
  \normalfont\selectfont}
\DeclareTextFontCommand{\textcyr}{\cyr}

\newcommand{\Cfvar}[1]{\text{C}^{#1}}
\newcommand{\Cfone}]{\ifmmode\Cfvar{1}\fi}
\newcommand{\Cftwo}]{\ifmmode\Cfvar{2}\fi}

\newcommand{\Cn}{\mathcal{C}(\R^n)}

\newcommand{\Dn}{\mathcal{D}(\R^n)}
\newcommand{\Dnminus}{\mathcal{D}(\R^{n-1})}

\newcommand{\Embed}[2]{J_{#1}(#2)}

\newcommand{\JacMat}[2]{J{#1}(#2)}

\newcommand{\Prj}{\Pi}

\newcommand{\QuLowerSubD}[2]{\underline{\partial} #1(#2)}

\newcommand{\QuUpperSubD}[2]{\overline{\partial} #1(#2)}

\newcommand{\Sn}{S_{n-1}}
\newcommand{\Snvar}[1]{S_{#1-1}}
\newcommand{\Snminus}{S_{n-2}}

\newcommand{\Stwo}{S_1}

\newcommand{\SubD}[2]{\partial #1(#2)}

\newcommand{\dirDeriv}[3]{#1'(#2;#3)}

\newcommand{\direct}[1]{\overrightarrow{#1}}

\newcommand{\hfre}{\hfill\mbox{}}

\newcommand{\oldminus}{\ominus}%{-}

\newcommand{\supf}[2]{\delta^*(#1,#2)}

\newcommand{\trans}{\top}

\def\transp{^{\rm T}}

\def\co{\mathop{\rm conv}}
\def\cl{\mathop{{\rm cl}}}

\renewcommand\int{\mathop{{\rm int}}}

\renewcommand{\a}{\alpha}
\renewcommand{\b}{\beta}
\renewcommand{\d}{\delta}

\renewcommand{\phi}{\varphi}

\newcommand{\myspan}{{\rm span\,}}

\newcommand{\dd}{{\overrightarrow \partial}}
\newcommand{\dc}{{\partial}} %%Convex subdifferential

\newcommand{\SphOne}{{{\cal S}_1}} %% One-dimensional sphere, i.e. { x\in R2: ||x||=1}
\newcommand{\Sphn}{{{\cal S}_{n-1}}} %% n-1-dimensional sphere, i.e. { x\in Rn: ||x||=1}

\newcommand{\Supp}{Y}
\newcommand{\SuppEl}{y}

\newcommand{\ch}{}
\newcommand{\chR}{}

\newtheorem{theorem}{Theorem}[section]
\newtheorem{corollary}{Corollary}[section]
\newtheorem{definition}{Definition}[section]
\newtheorem{exam}{Example}[section]
\newtheorem{lemma}{Lemma}[section]
\newtheorem{prop}{Proposition}[section]
\newtheorem{remark}{Remark}[section]

\newcounter{Examplecount}
\title{Directed Subdifferentiable Functions \protect\newline
       and the Directed Subdifferential \protect\newline
       \ch{without Delta-Convex} Structure}

\author{
Robert Baier%
\thanks{Chair of Applied Mathematics, University of Bayreuth,
                    95440 Bayreuth, Germany, robert.baier@uni-bayreuth.de}\;,
Elza Farkhi%
\thanks{School of Mathematical Sciences, Sackler Faculty of
                    Exact Sciences, Tel Aviv University, 69978 Tel Aviv, Israel, elza@post.tau.ac.il}\;
and Vera Roshchina\thanks{Collaborative Research Network, University of Ballarat, F Building,
                    Mount Helen Campus, PO BOX 663, Ballarat VIC 3353, Australia, vroshchina@ballarat.edu.au (Corresponding author)}
}

\begin{document}

\maketitle

\begin{abstract}

We show that the directed subdifferential introduced
%%% !!! priority of words before acronym
for differences of convex (\chR{delta-convex, DC}) functions by  Baier and Farkhi
%(AMS Contemp.Math. 513, 2010)
can be constructed from the directional
derivative without using any information on the delta-convex structure of the
function. The new definition extends to a more general class of functions, which
includes Lipschitz functions definable on o-minimal structure and quasidifferentiable functions.
\end{abstract}

\noindent {\bf \chR{Key Words}: }
   nonconvex subdifferentials,
   directional derivatives,
   difference of convex (\chR{delta-convex, DC}) functions, differences of sets

\noindent {\bf MSC: } 49J52, %%% optimization: nonsmooth analysis
                      90C26, %%% operations research, mathem. programming:
                             %%% nonconvex programming, global optimization
                      26B25, %%% real functions: convexity, generalizations
                      58C20 %%% global analysis, analysis on manifolds: differentiation theory (Gateaux, Fr\'{e}chet, etc.)

\section{Introduction}
\label{SecIntro}
The directed subdifferential was originally introduced in \cite{BaierFarkhiDC} for DC (delta-convex, i.e.~difference of convex) functions as the difference of the Moreau-Rockafellar subdifferentials of two convex functions embedded in the space of directed sets \cite{BaiFar2001,BaiFar2001b}, 
  %. Furthermore, as the directional derivative of  any quasidifferentiable (QD) function is the difference of the support functions of two convex compacts, the directed subdifferential 
and later was extended in \cite{BaiFarRosh2012a, BaiFarRosh2012b}
to the class of quasidifferentiable (QD) functions.

The directed subdifferential enjoys exact calculus rules, and gives the same sharp optimality conditions as the Demyanov-Rubinov quasidifferential or the
Dini subdifferential and superdifferential. However, in contrast to the quasidifferential, %which is an equivalence class of pairs of convex compact sets, the directed subdifferential , as a uniquely defined difference, does not depend on a chosen representating pair, 
the directed subdifferential is uniquely defined.
%, which eliminates the `inflation' in size when calculus rules are applied repeatedly. 
Moreover, one of the major challenges in quasidifferential calculus associated with the DC decomposition is not relevant for the directed
subdifferential: in this work we show that the directed subdifferential
can be constructed without any knowledge of the underlying DC structure.
In fact, the function (or its directional derivative) does not need to be DC. 

We extend the domain of definition of the directed subdifferential to a
large class of functions on $\R^n$, essentially wider than QD functions.
This class %of functions 
includes all locally Lipschitz functions definable on
o-minimal structure and tame functions (see \cite{Coste,DaniilidisPang,Iof2008}).
To our best knowledge, whether or not Lipschitz definable functions are DC constitutes
an open problem.

%However, our main achievement is that the new construction does not require any
%DC structure of the function or its directional derivative.
%What is only needed to calculate the directed subdifferential are (recursively)
%the directional derivatives of the function and of their restrictions to
%lower-dimensional affine sets.
%The directed subdifferential can be constructed for any $n$ times uniformly directionally differentiable function, %which we call `directed subdifferentiable'. The directed subdifferential is defined inductively in the dimension %$n$, analogously to the inductive definition of directed sets.

Let us note that other widely used subdifferentials can be evaluated via the directed
  subdifferential: Dini/Fr\'echet, Michel-Penot, and, in some cases also, Clarke and Mordukhovich subdifferentials (see \cite{BaiFarRosh2010}).
This together with the exact calculus rules for the directed subdifferential may help
 in practical applications where computing nonconvex subdifferentials is needed.

The paper is organized as follows: in Section~\ref{SecPreliminaries}, we remind
the basic ideas behind the Banach space of the directed sets, and,
in Section~\ref{SecDirSubdDC}, we briefly introduce the directed subdifferential,
if a DC structure for the
function or its directional derivative is known.
To give the reader the flavour of our results and to
explain our motivation, in Section~\ref{SecConstrForDC} we demonstrate
how to construct the directed subdifferential for DC functions without knowing
the DC decomposition. Section~\ref{SecDirSubdiff} is devoted to
`directed subdifferentiable' functions and a rigorous definition of the directed
 subdifferential for such functions is given there, as well as caculus rules and one example
 of a directed differentiable function which is not quasidifferentiable.

\section{Directed Sets -- a Brief Overview}\label{SecPreliminaries}

Here we only sketch the most important notions related to directed subdifferential. For more details and examples we refer the reader to \cite{BaiFar2001 , BaiFar2001b, BaierFarkhiDC, BaiFarRosh2012a, BaiFarRosh2012b}.

We use the standard notation in the paper: by $\|\cdot \|$ we denote the Euclidean norm, and $\Sn$ is the $(n-1)$-dimensional sphere in $\R^n$.

The directed sets  and the \emph{embedding} $J_n$
of convex compact sets in  $\R^n$ into the Banach space of directed sets
are introduced in \cite{BaiFar2001,BaiFar2001b} recursively in the % space
dimension $n$. The definition of the directed set is motivated by the Minkowski duality between compact convex sets and their support functions which correspond to the first and the second component in the definition.
Let $C \subset \R^n$ be nonempty, convex and compact and let $l \in \R^n$.
The \emph{support function} and respectively the \emph{supporting face} of $C$
in direction $l$ are defined by
\begin{eqnarray*}
   \supf{l}{C} & \ch{:=} & \max_{c \in C} \,\langle l,c\rangle\;, \\
   \Supp(l,C) & \ch{:=} &
       \{ y \in C \,|\, \langle l,y\rangle = \supf{l}{C}
       \}
   = \arg\max_{c \in C}\,\langle l,c\rangle\;.
\end{eqnarray*}
Note that $\Supp(\cdot,C)$ is upper semicontinuous and for $l=0$, \ $\Supp(l,C)=C$.
By $\SuppEl(l,C)$ we denote any point of the set $\Supp(l,C)$, and if the latter is a singleton (i.e.~there is a unique supporting point), then $\Supp(l,C)=\{\SuppEl(l,C)\}$.

In one dimension, the \emph{directed embedded intervals} are defined by the values of the support function in the two unit directions $\pm 1$.
Given an interval $[a,b]\in \R$, its embedded image in the space of
one dimensional directed space is
$$
  \overrightarrow{[a,b]}=
  J_1([a,b]) \chR{:=}
  (\delta^*(\eta,[a,b]))_{\eta = \pm 1}=(-a,b) \quad (a \leq b)\;.
$$

A general \emph{directed interval} $\overrightarrow{A}_1 = \overrightarrow{[c,d]} = (-c,d)$ allows
that $c,d$ are arbitrary real numbers, even $c > d$ is possible (see references in
\cite{BaiFar2001,BaiFar2001b}).

Linear operations on directed intervals are defined in a natural way and are motivated by the operations for vectors in $\R^2$: for directed intervals $\overrightarrow{[a,b]}$ and $\overrightarrow{[c,d]}$ and any $\a,\b\in \R$
\begin{equation}\label{eq:LinOpDirInt}
\alpha \overrightarrow{[a,b]}+\beta \overrightarrow{[c,d]} \ch{:=} \overrightarrow{[\alpha a+\beta c, \alpha b+\beta d]} \;.
\end{equation}

Even though the definition of the space of directed sets is rather abstract, it is motivated by the extension of the convex geometry: by embedding convex sets into the space of directed sets, it is possible to define a richer class of algebraic operations, including subtraction, which is exact and can be `inverted' (in contrast with the standard differences of convex sets, when the relation $(A-B)+B= A $ does not hold in general). We first give the (abstract) definition of the directed sets, and then give the geometrical motivation, discuss the embedding of the convex sets and the relevant operations on the directed sets.

\begin{definition}
   We call $\direct{A}$  a {\em directed set}
    \begin{itemize}
       \item[(i)]  in $\R$, iff it is a directed interval $\direct{A} \chR{:=} \direct{[-a_1(-1),a_1(1)]} =(a_1(-1), a_1(1))$. Its norm is
                    $\| \direct{A} \|_1 \chR{:=}
                   \max\limits_{ l = \pm 1 } | a_1(l) |$.
       \item[(ii)] in $\R^n,$ $n \geq 2$, iff there exist
                   a continuous function $a_n: \Sn \rightarrow \R$
                   and a uniformly bounded map having
                   lower-dimensional directed sets as images
                   $\direct{A_{n-1}}: \Sn \rightarrow
                   \Dnminus$ with respect to
                   $\| \cdot \|_{n-1}$.
    \end{itemize}
   We denote $\direct{A} \chR{:=} (\direct{A_{n-1}(l)}, a_n(l))_{l \in \Sn}$ and define its norm as
    $$
       \| \direct{A} \| :=
       \| \direct{A} \|_n := \max\{ \sup\limits_{ l \in \Sn }
                                    \| \direct{A_{n-1}(l)} \|_{n-1},
                                    \max\limits_{ l \in \Sn } | a_n(l) | \} \;.
    $$
   The set of all directed sets in $\R^n$ is denoted by $\Dn$.
\end{definition}
Observe that the space of directed sets is a Banach 
space (see \cite[Theorem~4.15]{BaiFar2001}).
%
%%% !!! slight reformulation to avoid overfull \hbox in next line:
%%% !!! The \emph{linear operations} are defined recursively on the two components of the
%%% !!! directed sets $\direct{A} = (\direct{A_{n-1}(l)}, a_n(l))_{l \in \Sn}$,
%%% !!! $\direct{B} = (\direct{B_{n-1}(l)}, b_n(l))_{l \in \Sn}$:
The \emph{linear operations} are defined recursively \chR{for the 
directed sets} $\direct{A} = (\direct{A_{n-1}(l)}, a_n(l))_{l \in \Sn}$,
$\direct{B} = (\direct{B_{n-1}(l)}, b_n(l))_{l \in \Sn}$
\chR{on each component}:
      \begin{equation}
       \begin{array}{|r@{\,}c@{\,}l|}
                  \hline \rule{0pt}{3ex}
          \direct{A} + \direct{B} & := & (\direct{A_{n-1}(l)} +
             \direct{B_{n-1}(l)}, a_n(l) + b_n(l))_{l \in \Sn} \;, \\[0.5ex]
          \lambda \cdot \direct{A} & := & (\lambda \cdot
             \direct{A_{n-1}(l)}, \lambda \cdot a_n(l))_{l \in \Sn}
             \hfre (\lambda \in \R) \;, \\[0.5ex]
          \direct{A} - \direct{B} & := & \direct{A} + (-\direct{B})
             = (\direct{A_{n-1}(l)} - \direct{B_{n-1}(l)},
                a_n(l) - b_n(l))_{l \in \Sn} \;. \\[0.5ex]
          \hline
       \end{array}
       \label{Eq_Operat_Dir_sets}
      \end{equation}

The space of directed sets is hence a normed vector space with zero\\ $\overrightarrow 0 \chR{:=} \overrightarrow{0_{\Dn}}\chR{:=}(0_{\Dn},0)_{l\in \Sn}$ (see \cite[Proposition~4.14]{BaiFar2001}).

The \textit{embedding} $J_n: \Cn \rightarrow \Dn$ of the convex sets in the space of directed sets which determines for every
set $A \in \Cn$ its embedded image $\direct{A} \in \Dn$
is defined as follows:
 \begin{itemize}
    \item[(i)]  For $n = 1$, \ $J_1([a,b]) := \chR{\direct{[a,b]} :=} (-a,b)$ \chR{for $a \leq b$}.
    \item[(ii)] For $n \ge 2$, \ $J_n(A) := \chR{\direct{A}:=} \big(J_{n-1}(\Prj_{n-1,l}(Y(l,A))),\delta^*(l,A)\big)_{l \in \Sn}$, where
                \rule{0ex}{2.5ex}
                $J_{n-1}(\Prj_{n-1,l}(Y(l,A)))$
                is an $n-1$-dimensional image of the face $Y(l,A)$,
                $\Prj_{n-1,l}(x) := \pi_{n-1,n} R_{n,l}(x)$ and  $\pi_{n-1,n} \in
                        \R^{(n-1) \times n}$ is the natural projection,
                i.e.~$\pi_{n-1,n}(x_1,\ldots,x_{n-1},x_n) \chR{:=} (x_1,\ldots,x_{n-1})$,
                and $R_{n,l}$ is a fixed rotation for every $l \in \Sn$
		  satisfying
                 \begin{equation}
                    R_{n,l}(l) = e^n \;, \qquad
                    R_{n,l}(\mbox{span}\{l\}^\perp) =
                      \mbox{span}\{ e^1, e^2, \ldots, e^{n-1} \} \;.
                  \label{rotate}
                 \end{equation}
 \end{itemize}
We also note that the embedding is positively linear by \cite[Proposition~3.10
and Theorem~4.17]{BaiFar2001}, i.e.
 \begin{align} \label{eq:embed_pos_lin}
    \Embed{n}{ \lambda A + \mu B }
      & = \lambda \Embed{n}{ A } + \mu \Embed{n}{ B }
        \quad \text{($A,B \in \Cn$, $\lambda,\mu \geq 0$)}.
 \end{align}

\section{The Directed Subdifferential \protect\newline with Known \chR{Delta-Convex} Structure}
\label{SecDirSubdDC}

Before introducing the directed subdifferential, let us recall that the classical Moreau-Rockafellar \textit{convex  subdifferential}
of a convex function $f: \R^n \rightarrow \R$ at $x \in \R^n$ is a compact convex set
 \begin{equation}
       \dc f(x) := \{ s \in \R^n \,|\, \forall y \in \R^n:\
                         \langle s,y-x \rangle + f(x) \leq f(y) \}\;.
     \label{Equ_conv_subdiff_1}
 \end{equation}
In the sequel, the Moreau-Rockafellar {\em convex subdifferential
of a sublinear function} $g$ at zero is denoted by $\dc g$ instead of $\dc g(0)$.

Note that $\Supp(l,C)$ equals the subdifferential of the support
function of $C$ at $l$ \cite[Corollary~23.5.3]{Roc1972}.

It is well-known (see e.g.~\cite[Chap.~VI, Definition~1.1.4 and Example~3.1, (3.1)]{HirUrrLem1993}) that
\begin{align}
  \supf{l}{\dc f(x)} & =f'(x; l)\;, \label{12'} \\
  \SubD{ \supf{\cdot}{A} }{l} & = \Supp(l, A) \quad (A \in \Cn)\;, \label{12'''}
    \\
  \intertext{and hence,}
  \SubD{ \dirDeriv{f}{x}{\cdot} }{l} & = \Supp(l, \dc f(x)) \;. \label{12''}
\end{align}
where $f'(x; l)$ is the (Dini) \emph{directional derivative} of $f$ at $x$ in a direction $l$,
$$
f'(x;l) \ch{:=} \lim_{\ch{t\downarrow 0}} \frac{f(x+tl)-f(x)}{t} \;.
$$

The following statement is \cite[Proposition~3.1 from Chap.~I.3]{DemRub1995}
and is needed later for proving the closure of arithmetic operations for
directed subdifferentiable functions:

\begin{remark}\label{rem:DirDer} Let $f_1,f_2:\R^n\to \R$ be directionally differentiable at $x\in \R^n$. Then their sum, product and quotient (if $f_2(x) \ne 0$) are also directionally differentiable
at this point and the following formulas hold for $\alpha,\beta \in \R$:
\begin{align*}
(\alpha f_1+\beta f_2)'(x;l) & = \alpha f'_1(x;l)+\beta f'_2(x;l) \;,\\
(f_1\cdot f_2)'(x;l) & = f_1(x)f'_2(x;l)+f_2(x)f'_1(x;l) \;,\\
\left(\frac{f_1}{f_2}\right)'(x;l) & = -\frac{f_1(x)f'_2(x;l)-f_2(x)f'_1(x;l)}{[f_2(x)]^2}
\end{align*}
\end{remark}

A similar closedness property with respect to min/max operations is given in \cite[Corollary~3.2 in Sect.~I.3]{DemRub1995}:

\begin{remark}\label{rem:DirDerMaxMin} \sloppy{Let $f_i:\R^n\to \R$, $i\in I = \{1,2,\dots, p\}$ have finite Dini directional derivatives at $x\in \R^n$ for all $l\in \Sphn$. Then the pointwise minimum and pointwise maximum functions ${f_{\max}(\cdot) \chR{:=} \max_{i\in I}f_i(\cdot)}$ and ${f_{\min}(\cdot) \chR{:=} \min_{i\in I} f_i(\cdot)}$ are also Dini directionally differentiable at $x$ with}
$$
f_{\max}'(x;l) = \max_{i\in I(x)} f_i'(x;l);\quad f_{\min}'(x;l) = \min_{i\in J(x)} f_i'(x;l) \;,
$$
where $I(x) \chR{:=} \{i\in I\,|\, f_i(x) = f_{\max}(x)\}$, $J(x) \chR{:=} \{i\in I\,|\, f_i(x) = f_{\min}(x)\}$.
\end{remark}

Clearly, we can embed the Moreau-Rockafellar subdifferential in the space of directed sets, i.e.
 \begin{align} \label{eq:embed_conv_subd}
    \dd f(x)
      & = \begin{cases}
             \big( \supf{l}{ \SubD{f}{x} } \big)_{l = \pm 1}
               & \ \, \text{if $n = 1$}, \\
             \big( \Embed{n-1}{\Prj_{n-1,l}(\Supp(l, \SubD{f}{x}))},
                \supf{l}{ \SubD{f}{x} }
             \big)_{l \in \Sn} & \ \, \text{if $n \geq 2$}.
          \end{cases}
 \end{align}
Moreover, using the embedding, we can define the directed subdifferential of a DC function \cite{BaierFarkhiDC}.

\begin{definition}
\label{Def_Dir_subdif}
   Let $f: \R^n \rightarrow \R$ be DC, i.e.~$f=g-h$, where $g,h:\R^n\to \R$ are convex.
   The {\em directed subdifferential} of $f$ at $x$ is defined by
    \begin{align}\label{eq:defDCdd}
       \dd_{DC} f(x) & \ch{:=} \Embed{n}{ \dc g(x)}
                        - \Embed{n}{ \dc h(x)} \;.
    \end{align}
\end{definition}
The directed subdifferential was extended to the larger class of
quasidifferentiable functions in \cite{BaiFarRosh2012a, BaiFarRosh2012b}.
The core idea of this generalization is that instead of DC functions more general functions are considered, whose directional derivatives are DC (see e.g.~\cite[Chap.~3]{DemRub1995}
and \cite[Sect.~10.2]{PallaUrban2002}).
The directed subdifferential is then defined via the directional derivative.

\begin{definition}
\label{Def_Dir_subdif_QD}
   Let $f: \R^n \rightarrow \R$ be \emph{quasidifferentiable (QD)}, i.e.
    \begin{align*}
       f'(x; l) & = \supf{l}{ \QuLowerSubD{f}{x} }
         - \supf{l}{ \oldminus \QuUpperSubD{f}{x} }
         \quad (l \in \Sn)
    \end{align*}
   with two convex, compact, nonempty sets $\QuLowerSubD{f}{x},
   %%% !!! use of "i.e.~" instead of ":", since afterwards ":=" is used
   \QuUpperSubD{f}{x} \subset \R^n$. Here\chR{,} $\oldminus C$ is the algebraic negative of the convex set $C$\chR{, i.e.~}$\oldminus C \ch{:=} \{-x,\,|\, x\in C\}$.\\
   The {\em directed subdifferential} of $f$ at $x$ is defined by
    \begin{align}\label{eq:defQDdd}
       \dd_{QD} f(x) & \ch{:=} \Embed{n}{ \QuLowerSubD{f}{x} }
                        - \Embed{n}{ \oldminus \QuUpperSubD{f}{x} } \;.
    \end{align}
\end{definition}
Since
 \begin{align*}
    \QuLowerSubD{f}{x} & = \dc g(x) \;, \
    \QuUpperSubD{f}{x} = \oldminus \dc h(x)
 \end{align*}
for a DC function $f = g-h$, Definition~\ref{Def_Dir_subdif_QD}
is an extension of the directed subdifferential for DC functions introduced
in Definition~\ref{Def_Dir_subdif}, namely it holds for DC functions
 \begin{align} \label{eq:dir_subd_ext}
    \dd_{QD} f(x) & = \dd_{DC} f(x) = \dd f(x) \;.
 \end{align}
We would like to turn our attention to quasidifferentiable functions and
state the following lemma which is a slight extension of
\cite[Proposition~3.15]{BaiFarRosh2010} from convex
to quasidifferentiable functions.

\begin{lemma}\label{lem:ddQD} Let $f:\R^n\to \R$ be quasidifferentiable. Then the directed subdifferential of $f$ fulfills
\begin{equation}\label{eq:ddQD01}
\dd_{QD} f(x) = \dd_{DC} f'(x;\cdot)(0) \;.
\end{equation}
\end{lemma}
\begin{proof} \sloppy{Observe that by the definition of a QD function, its directional derivative $f'(x;\cdot)$ is DC, i.e.~$f'(x;l) = g(l)-h(l)$, with ${g(l) = \delta^*(l,\QuLowerSubD{f}{x})}$, ${h(l) = \delta^*(l,\oldminus \QuUpperSubD{f}{x})}$; hence, $\partial g(0) = \QuLowerSubD{f}{x}$, $\partial h(0) = \oldminus \QuUpperSubD{f}{x}$. From \eqref{eq:defDCdd} we have}
$$
\dd_{DC} f'(x;\cdot)(0)  = \Embed{n}{ \dc g(x)} - \Embed{n}{ \dc h(x)} 
= \Embed{n}{\QuLowerSubD{f}{x}} - \Embed{n}{\oldminus \QuUpperSubD{f}{x}},
$$
which coincides with \eqref{eq:defQDdd}.

%%     Setting $\dirDeriv{f}{x}{l} = \widetilde{g}(l) - \widetilde{h}(l)$ with
%%     %
%%      \begin{align*}
%%         \widetilde{g}(l) & = \supf{l}{ \QuLowerSubD{f}{x} } \;, \
%%           \widetilde{h}(l) = \supf{l}{ \oldminus \QuUpperSubD{f}{x} } \;,
%%      \end{align*}
%%     %
%%     we can use~\eqref{12'''} which yields
%%     %
%%      \begin{align*}
%%         \dc \widetilde{g}(l) & = \Supp(l, \QuLowerSubD{f}{x}) \;, \
%%           \dc \widetilde{h}(l) = \Supp(l, \oldminus \QuUpperSubD{f}{x}) \;.
%%      \end{align*}
%%     %
%%     Theorem~\ref{prop:dir_subd_dc} applies and we can prove that
%%     %
%%      \begin{align*}
%%         \dd_{DC} f'(x;\cdot)(0) & = \Embed{n}{ \dc \widetilde{g}(x)}
%%           - \Embed{n}{ \dc \widetilde{h}(x)} \\
%%         & = \Embed{n}{ \Supp(0, \QuLowerSubD{f}{x}) }
%%           - \Embed{n}{ \Supp(0, \oldminus \QuUpperSubD{f}{x}) } \\
%%         & = \Embed{n}{ \QuLowerSubD{f}{x} }
%%           - \Embed{n}{ \oldminus \QuUpperSubD{f}{x} } = \dd_{QD} f(x) \;.
%%      \end{align*}
%%     %
\end{proof}

The previous result allows \ch{one} to carry over results on the directed subdifferential
from the DC to the quasidifferentiable functions.

\section{The Directed Subdifferential \protect\newline
         without a \chR{Delta-Convex} Structure}
\label{SecConstrForDC}

In this \ch{section, we} suggest a technique for constructing the directed subdifferential only  from the directional derivative, i.e.~without using any information on the DC structure of the function and avoiding the direct embedding of the relevant subdifferentials. For the sake of clarity, we start with convex functions, then extend the construction to DC functions, and end up with defining a more general class of `directed subdifferentiable' functions.

Starting from the equality~\eqref{eq:embed_conv_subd} of the embedded
subdifferential of convex analysis,
we now obtain an explicit representation of the directed subdifferential of a convex function via its directional derivatives. Here, we can apply~\eqref{12'} to
see that the second component of the directed subdifferential
in~\eqref{eq:embed_conv_subd} equals the directional derivative.
To obtain an explicit formula for the first component as well, we use
an auxiliary function $f_l(\cdot)$ which depends only on $n-1$ variables.

We first gather the result for convex functions and apply the new
representation to DC and quasidifferentiable functions in a second step.

\subsection{Convex Functions}

Recall that the directed subdifferential of a convex function is the embedded image of its Moreau-Rockafellar subdifferential into the space of directed sets.
We describe its construction first in one dimension.

It is not difficult to observe that for any convex function $f:\R\to \R$ with 
%%% !!! do not define $[a,b]$ as subdifferential \dots
%%% !!! $[a, b] : = \dc f(x) \chR{=[a,b]}$ we have
$\dc f(x) \chR{=[a,b]}$ we have
$$
f'(x;-1) = \max_{v\in [a,b]}\{-v\}= -a;\quad f'(x;1) = \max_{v\in [a,b]}\{v\} = b \;,
$$
and hence, by the definition of the directed subdifferential,
\begin{equation}\label{eq:ddd_c1d}
\dd_{DC} f(x) = J_1(\dc f(x)) = \overrightarrow{[a,b]}=\overrightarrow{[-f'(x;-1),f'(x;1)]} \;.
\end{equation}

Based on this construction, we now obtain a similar representation for higher dimensions.  Consider a convex function $f:\R^n\to \R$, and fix a point $x\in \R^n$. Let $f'(x;\cdot)$ be the directional derivative of $f$ at $x$. To define an explicit expression of the first term in the right-hand side of \eqref{eq:embed_conv_subd} for a given vector $l\in \Sn$ we define an auxiliary function $f_l:\R^{n-1}\to \R$ which is
a restriction of $f'(x;\cdot)$  to the hyperplane $l+ \myspan \{l\}^\perp$, i.e.~for $y\in \R^{n-1}$ let
\begin{equation}\label{eq:fxl}
f_l(y) \ch{:=} f'(x;l+\Prj\transp_{n-1,l}y) \;,
\end{equation}
%%% !!! don't define transposed of mapping
%%% !!! where $\Prj\transp_{n-1,l}\ch{:=} R^{-1}_{n,l}\pi\transp_{n-1,n}$ maps $\R^{n-1}$ to $\myspan \{l\}^\perp$. 
where \chR{$\Prj_{n-1,l}\ch{:=} \pi_{n-1,n} R_{n,l}$ and $\Prj\transp_{n-1,l}$} maps $\R^{n-1}$ to $\myspan \{l\}^\perp$. 
Here\chR{,} $\pi\transp_{n-1,n}$ maps points in $\R^{n-1}$ to $\R^n$ by adding a zero last coordinate,
%%% !!! delete "and" to avoid a linebreak in formula of $R^{-1}_{n,l}$:
\chR{$R^{-1}_{n,l}:\R^n\to \R^n$} is the inverse of the fixed rotation introduced in Section~\ref{SecPreliminaries} (see~\eqref{rotate}). Thus,
$$
R^{-1}_{n,l}(e^n) = l, \quad R^{-1}_{n,l} ({\myspan}\{e_1,e_2,\dots, e_{n-1}\}) = {\myspan}\{l\}^\perp \;.
$$
In the following
proposition the first (lower-dimensional) component of the directed subdifferential is expressed via a directional derivative of a certain function
restricted to an affine space.
\begin{prop} \label{prop:dir_subd_conv} Let $f:\R^n\to \R$ be convex. Then,
\begin{equation}\label{eq:ddn}
\dd_{DC} f(x) = \begin{cases}
              \left( f'(x;l) \right)_{l = \pm 1} & \quad\text{for $n=1$}, \\
              \left(\dd_{DC} f_l(0),f'(x;l)\right)_{l\in \Sn}
                & \quad\text{for $n \geq 2$}.
           \end{cases}
\end{equation}
\end{prop}
\begin{proof} Taking~\eqref{eq:ddd_c1d} into account, we only prove the case $n \geq 2$.

First observe that the directional derivative $\dirDeriv{f}{x}{l}$ of a convex function is also convex as a function of the direction $l$, therefore $\dirDeriv{f}{x}{l}$ is
directionally differentiable, and the higher directional derivatives as well. Hence, in the case of convex functions directional derivatives of any `order' are defined.

Since $f$ is convex, we can apply~\eqref{12'}--\eqref{12''} and it follows that
\begin{align*}
f'_l(0;u)
    & = \lim_{\ch{t\downarrow 0}}\frac{f_l(tu)-f_l(0)}{t}\\
    & = \lim_{\ch{t\downarrow 0}}\frac{f'(x;l+t \Prj\transp_{n-1,l} u)-f'(x;l)}{t}\\
    & = [f'(x;\cdot)]'(l;\Prj\transp_{n-1,l} u) \;.%\\
\end{align*}
    Using the equality
 $[f'(x;\cdot)]'(l;v)=\d^*(v;Y(l,\dc f(x)))$ (see
 e.g.~\cite[Lemma~3.4]{BaiFarRosh2012a}), we get
 \begin{align*}
f'_l(0;u)
    & = \max_{\zeta \in Y(l,\dc f(x))} \zeta\transp \Prj\transp_{n-1,l}u\\
    & = \max_{\zeta\in Y(l,\dc f(x))} (\Prj_{n-1,l} \zeta)\transp u\\
    & =\max_{\xi\in \Prj_{n-1,l} Y(l,\dc f(x))} \xi\transp u
      =\d^*(u;\Prj_{n-1,l} Y(l,\dc f(x))) \;.
\end{align*}
From here, by~\eqref{12'} we get $\Prj_{n-1,l} Y(l,\dc f(x)) = \dc f_l(0)$, hence
$$
\dd_{DC} f_l(0) = J_{n-1}(\Prj_{n-1,l} Y(l;\dc f(x))) \;.
$$
This, together with \eqref{12'} and \eqref{eq:embed_conv_subd}
implies \eqref{eq:ddn}.
\end{proof}

We illustrate our construction with a simple example.

\begin{exam}\label{ex:DSConvex}
Consider the function $f:\R^2\to \R$, $f(x) \chR{:=} \max\{|x_1|,|x_2|\}$, and for any $l\in \SphOne$ let
$$
R_{2,l} = \left[\begin{array}{rr}l_2  & - l_1\\ l_1 & l_2\end{array}\right] \;, \mbox{ then } R^{-1}_{2,l} = \left[\begin{array}{rr}l_2  &  l_1\\ -l_1 & l_2\end{array}\right] \;.
$$
We first construct the directed subdifferential by embedding the subdifferential of $f$ into the space of the directed sets as in the original definition of the directed subdifferential, and then show that the same result can be obtained via the directional derivative approach.

Since $f$ is positively homogeneous, $f'(0;p) = f(p)$ for all $p\in \R^2$; besides it is not difficult to see (e.g.~applying \cite[Corollary~4.3.2 from Chap.~VI]{HirUrrLem1993}) that the grey polytope in Fig.~\ref{fig:Example1FirstMethod}(a) equals
$$
\dc f = \dc f(0) = \co\{(1,0), (0,1), (-1,0), (0,-1)\} \;.
$$
By equality~\eqref{eq:embed_conv_subd} we have
\begin{equation}\label{eq:001}
\dd_{DC} f(0) = J_2(\dc f) = \left(J_1(\Prj_{1,l}(Y(l,\dc f))), \d^*(l;\dc f)\right)_{l\in \SphOne} \;.
\end{equation}
It is not difficult to observe that for $l\in \SphOne$
$$
Y(l,\dc f) = \left\{
    \begin{array}{ll}
        \{(1,0)\} \;, & l_1>|l_2| \;,\\
        \{(-1,0)\} \;, & l_1<-|l_2| \;,\\
        \{(0,1)\} \;, & l_2>|l_1| \;,\\
        \{(0,-1)\} \;, & l_2<-|l_1| \;,\\
        \co \{(1,0), (0,1)\} \;, & l = (\frac{1}{\sqrt{2}},\frac{1}{\sqrt{2}}) \;,\\
        \co \{(-1,0), (0,1)\} \;, & l = (-\frac{1}{\sqrt{2}},\frac{1}{\sqrt{2}}) \;,\\
        \co \{(-1,0), (0,-1)\} \;, & l = (-\frac{1}{\sqrt{2}},-\frac{1}{\sqrt{2}}) \;,\\
        \co \{(1,0), (0,-1)\} \;, & l = (\frac{1}{\sqrt{2}},-\frac{1}{\sqrt{2}}) \;.
    \end{array}
\right.
$$
Since
$$
\Prj_{1,l}=\pi_{1,2}R_{2,l} = [1,0] \left[\begin{array}{rr}l_2  & - l_1\\ l_1 & l_2\end{array}\right] = [l_2,-l_1] \;,
$$
we have for $l\in \SphOne$
\begin{equation}\label{eq:DlviaJ1}
\overrightarrow D_l: = J_1(\Prj_{1,l}(Y(l,\dc f))) = \left\{
    \begin{array}{ll}
        \overrightarrow{\{l_2\}} \;, & l_1>|l_2| \;,\\
        \overrightarrow{\{-l_2\}} \;, & l_1<-|l_2| \;,\\
        \overrightarrow{\{-l_1\}} \;, & l_2>|l_1| \;,\\
        \overrightarrow{\{l_1\}} \;, & l_2<-|l_1| \;,\\
        \overrightarrow{\left[-\frac{1}{\sqrt{2}},\frac{1}{\sqrt{2}}\right]} \;, & |l_1|=|l_2| \;.
    \end{array}
\right.
\end{equation}
The construction of the directed subdifferential is illustrated on Fig.~\ref{fig:Example1FirstMethod}.

\begin{figure}[h]
\centering \includegraphics[height=120pt, keepaspectratio ]{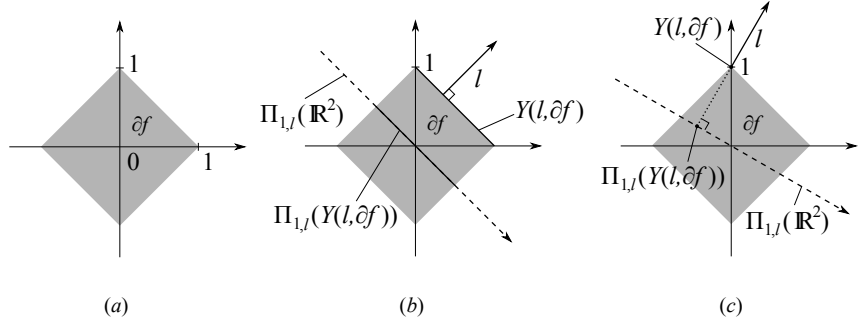}
\caption{Construction of the directed subdifferential via embedding of $\partial f$ (Example~\ref{ex:DSConvex}): a) the subdifferential of $f$ at $0_2$; b) projecting the supporting faces onto the one-dimensional subspace $\Pi_{1,l}(\R^2)$ for  $l=(1/\sqrt{2},1/\sqrt{2})$ and c) for $l = (1/2, \sqrt{3}/2)$}
\label{fig:Example1FirstMethod}
\end{figure}

\ch{Noting} that by $f'(0;l) = \d^*(l;\dc f)$, we have
\begin{equation}\label{eq:egDDviaEmb}
\dd_{DC} f(0) = \left(\overrightarrow D_l, f'(0;l)\right)_{l\in \SphOne}
\end{equation}
from \eqref{eq:001} and \eqref{eq:DlviaJ1},
where $\overrightarrow D_l$ is given by \eqref{eq:DlviaJ1}.

We now obtain the directed subdifferential via the directional derivative based procedure. Observe that
for every $y\in \R$
$$
l+\Prj\transp_{1,l}y
=
        \left[\begin{array}{r} l_1\\ l_2\end{array}\right]
        + \left[\begin{array}{rr}l_2  &  l_1\\ -l_1 & l_2\end{array}\right]  \left[\begin{array}{r} 1 \\ 0\end{array}\right]y\\
=
        \left[\begin{array}{r} l_1+ yl_2\\ l_2-yl_1\end{array}\right] \;,
$$
hence,
$$
f_l(y) = f'(0;l+\Prj\transp_{1,l}y) = f(l+\Prj\transp_{1,l}y) = \max\{|l_1+yl_2|, |l_2-yl_1|\} \;.
$$
An illustration of how to construct $f_l$ is given in Fig.~\ref{fig:Example1SecondMethod}.

\begin{figure}[ht]
\centering \includegraphics[height=200pt, keepaspectratio ]{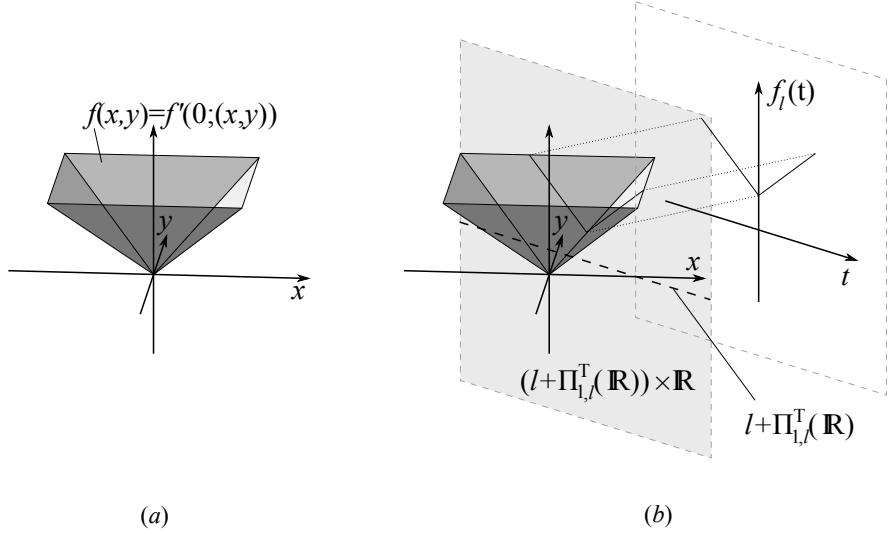}
\caption{Construction of the restriction $f_l$ of the function $f$ (Example~\ref{ex:DSConvex})
for $l=(1/\sqrt{2}, 1/\sqrt{2})$: a) graph of the function $f$ which in this case coincides with $f'(0_2;\cdot)$; b) the restriction of $f$ to $l+\Prj\transp_{1,l}(\R)$} \label{fig:Example1SecondMethod}
\end{figure}

It is a trivial exercise to compute
$$
f_l'(0;1) = \left\{
  \begin{array}{rr}
        l_2 \;, & l_1>|l_2| \;,\\
        -l_2 \;, & l_1<-|l_2| \;,\\
        -l_1 \;, & l_2>|l_1| \;,\\
        l_1 \;, & l_2<-|l_1| \;,\\
        \frac{1}{\sqrt{2}} \;, & |l_1|=|l_2| \;,
  \end{array}
\right.
\qquad
f_l'(0;-1) = \left\{
  \begin{array}{rr}
        -l_2 \;, & l_1>|l_2| \;,\\
        l_2 \;, & l_1<-|l_2| \;,\\
        l_1 \;, & l_2>|l_1| \;,\\
        -l_1 \;, & l_2<-|l_1| \;,\\
        \frac{1}{\sqrt{2}} \;, & |l_1|=|l_2| \;.
  \end{array}
\right.
$$

Applying \eqref{eq:ddd_c1d} to compute $ \dd_{DC} f_l(0)$, we observe that $\dd_{DC} f_l(0) = \overrightarrow D_l$, where $\overrightarrow D_l$ is defined by \eqref{eq:DlviaJ1}. Therefore, by \eqref{eq:ddn}
$$
\dd_{DC} f(x) = \left(\dd_{DC} f_l(0),f'(x;l)\right)_{l\in \SphOne}=\left(\overrightarrow D_l,f'(x;l)\right)_{l\in \SphOne} \;,
$$
which coincides with \eqref{eq:egDDviaEmb}.
\end{exam}

\ch{In next} section we show that the directed subdifferential of a DC and a QD function can also be constructed via directional derivatives,
i.e.~Proposition~\ref{prop:dir_subd_conv} also holds for DC and QD functions.

\subsection{\chR{Delta-Convex} and Quasidifferentiable Functions}\label{Subsec:DCfunctons}

Assume that $f:\R\to\R$ is DC with $f=g-h$ (where $g,h:\R\to \R$ are convex), then for any $x\in \R$ using \eqref{eq:ddd_c1d} we get
\begin{align}
\dd_{DC} f (x) & = J_1(\dc g(x))- J_1(\dc h(x)) \nonumber \\
          & = \overrightarrow{[-g'(x;-1),g'(x;1)]} -\overrightarrow{[-h'(x;-1),h'(x;1)]} \nonumber \\
          & = \overrightarrow{[-(g'(x;-1)-h'(x;-1)),g'(x;1)-h'(x;1)]} \nonumber \\
          & = \overrightarrow{[-f'(x;-1),f'(x;1)]} \;, \label{eq:dir_subd_dc_1d}
\end{align}
%%% !!! fill words "the directional derivative" to avoid big horizontal gaps
%%% !!! different page breaks until p. 19 due to additional line, but I think this is okay
\sloppy{since \chR{the directional derivative} $f'(x;\cdot) = g'(x;\cdot) - h'(x;\cdot)$ (see e.g.~Remark~\ref{rem:DirDer}).
Hence, ${\dd_{DC} f(x) = \overrightarrow{[-f'(x;-1),f'(x,1)]}}$, which coincides with the expression \eqref{eq:ddd_c1d}.}

Now let $f:\R^n\to \R$ be DC, and fix a point $x\in \R^n$. Let $f'(x;\cdot)$ be the directional derivative of $f$ at $x$. Consider the restriction $f_l$ of the directional derivative defined by \eqref{eq:fxl}.

\begin{theorem} \label{prop:dir_subd_dc}
  Let $f:\R^n\to \R$ be DC with $f = g - h$ with convex functions
  $g,h: \R^n\to \R$. Then,
\begin{equation}\label{eq:ddn_dc}
\dd_{DC} f(x) = \begin{cases}
              \left( f'(x;l) \right)_{l = \pm 1} & \quad\text{for $n=1$}, \\
              \left(\dd_{DC} f_l(0),f'(x;l)\right)_{l\in \Sn}
                & \quad\text{for $n \geq 2$}.
           \end{cases}
\end{equation}
\end{theorem}
\begin{proof} Since~\eqref{eq:dir_subd_dc_1d} is already proved, we only
consider the case $n \geq 2$.

By Remark~\ref{rem:DirDer} $f'(x;l) = g'(x;l)-h'(x;l)$, and hence for any $y\in \R^{n-1}$
$$
f_l(y)  = g'(x;l+\Prj\transp_{n-1,l}y)-h'(x;l+\Prj\transp_{n-1,l}y) = g_l(y)-h_l(y) \;.
$$
Since both $g$ and $h$ are convex, we can use the results from the previous section, i.e.~$\dc g_l = \Prj_{n-1,l} Y(l;\dc g(x))$ and $\dc h_l = \Prj_{n-1,l} Y(l;\dc h(x))$.
Observe that $f_l$ is also DC, since $g_l$ and $h_l$ are convex and $f_l = g_l-h_l$. Therefore, the directed subdifferential of $f_l$ at $0$ is the difference of the convex embeddings of $\dc g_l$ and $\dc h_l$:
\begin{align}\label{eq:dddAn}
\dd_{DC} f_l  & = J_{n-1} (\dc g_l) - J_{n-1} (\dc h_l)\notag \\
         & = J_{n-1} (\Prj_{n-1,l} Y(l;\dc g(x))) - J_{n-1}(\Prj_{n-1,l} Y(l;\dc h(x))) \;;
\end{align}
moreover,
\begin{equation}\label{eq:dddan}
\d^*(l;\dc g(x)) - \d^*(l;\dc h(x)) = g'(x;l) - h'(x;l) = f'(x;l) \;.
\end{equation}
Now comparing \eqref{eq:dddAn} and \eqref{eq:dddan} with the definition of the directed subdifferential, we obtain \eqref{eq:ddn_dc}.
\end{proof}

The case of quasidifferentiable functions is a consequence of the DC case
by taking into account Lemma~\ref{lem:ddQD}:

\begin{corollary} \label{cor:dir_subd_qd}
  Let $f:\R^n\to \R$ be quasidifferentiable with
  $$
     f'(x;l) = \supf{l}{ \QuLowerSubD{f}{x} }
               - \supf{l}{ \oldminus \QuUpperSubD{f}{x} }
     \quad (l \in \Sn)
  $$
  and $\QuLowerSubD{f}{x}, \QuUpperSubD{f}{x} \in \Cn$.
  Then,
\begin{equation}\label{eq:ddn_qd}
\dd_{QD} f(x) = \begin{cases}
              \left( f'(x;l) \right)_{l = \pm 1} & \quad\text{for $n=1$}, \\
              \left(\dd_{QD} f_l(0),f'(x;l)\right)_{l\in \Sn}
                & \quad\text{for $n \geq 2$}.
           \end{cases}
\end{equation}
\end{corollary}
\begin{proof}
   Let $\varphi(l) \chR{:=} f'(x; l)$, then $\varphi(\cdot)$ is a DC function.
   It is easy to conclude that
    \begin{align*}
       \varphi(0) & = f'(x; 0) = 0 \;, \\
       \varphi(h \eta) & = f'(x; h \eta) = h f'(x; \eta) = h \varphi(\eta) \;, \\
       \varphi'(0; \eta) & = \lim_{h \downarrow 0} \frac{1}{h} (\varphi(h \eta)
         - \varphi(0)) = \varphi(\eta) = f'(x; \eta) \;, \\
       \varphi_l(0) & = \varphi'(0; l + \Prj_{n-1,l}\transp \eta)
         = f'(x; l + \Prj_{n-1,l}\transp \eta) = f_l(0) \;.
    \end{align*}
   Theorem~\ref{prop:dir_subd_dc} and Lemma~\ref{lem:ddQD} apply which show
    \begin{align*}
       \dd_{DC} \varphi(0)
         & = \begin{cases}
                \left( \varphi'(0;l) \right)_{l = \pm 1} & \quad\text{for $n=1$}, \\
                \left(\dd_{DC} \varphi_l(0),\varphi'(0;l)\right)_{l\in \Sn}
                  & \quad\text{for $n \geq 2$}
             \end{cases} \\
       & = \begin{cases}
              \left( f'(x;l) \right)_{l = \pm 1} & \quad\text{for $n=1$}, \\
              \left(\dd_{DC} f_l(0), f'(x;l)\right)_{l\in \Sn}
                & \quad\text{for $n \geq 2$}
           \end{cases} \\
       & = \dd_{QD} f(x) \;.
    \end{align*}
\end{proof}

The previous Theorem~\ref{prop:dir_subd_dc} and Corollary~\ref{cor:dir_subd_qd}
can be seen as generalizations of~\cite[Lemma~3.10 and Theorem~3.11]{BaiFarRosh2010}
to the case of difference of two sublinear functions~(DS) functions
$f: \R^n \to \R$
for state dimensions $n \geq 2$ and to DC/quasidifferentiable functions $f: \R^n \to \R$ for $n \in \N$
(in the additional view of \cite[Lemmas~3.6, 3.8 and~3.9]{BaiFarRosh2010}).
In~\cite[Theorem~3.11]{BaiFarRosh2010} the directed subdifferential for such
functions is expressed with the limes superior of Fr\'{e}chet subdifferentials
as first component and the directional derivative as second component (see
\cite{BaiFarRosh2010} for more details).
But in the latter article, the primal aim was to discover the connection of the
Mordukhovich/basic subdifferentials to the Rubinov subdifferential
(the visualization of the directed subdifferential).

Hence, we can give an equivalent definition of the directed subdifferential via directional derivatives for both cases (DC and quasidifferentiable functions). Moreover, it appears that  in contrast to the original definitions of the directed subdifferential in \cite{BaierFarkhiDC,BaiFarRosh2012a}, these results lead to a new definition of the directed subdifferential for
an extended class of functions wider than DC or QD functions. We first need to define such functions.

\section{Extension of the Directed Subdifferential}
\label{SecDirSubdiff}

In this section, we give a definition of `directed subdifferentiable' functions and show that a directed subdifferential can be defined for such functions even though they are not necessarily DC (and even not quasidifferentiable).

\subsection{Directed Subdifferentiable Functions}\label{SecDefDirSubFunc}
We define the `directed subdifferentiable' functions recursively by considering restrictions of directional derivatives to affine subspaces in a recursive way.
\begin{definition}\label{def:dsf}{\bf (directed subdifferentiable function)} Given a constant $M\geq 0$, a function $f:\R \to \R$ is called $M$-directed subdifferentiable at $x\in \R$, iff both its left- and right-sided derivatives exist and
$$
   \max \{|f'(x;-1)|,|f'(x;1)|\}\leq M \;.
$$
A function $f:\R^n\to \R$, $n\ge 2$ is called $M$-directed subdifferentiable at $x\in \R^n$, iff its directional derivative $f'(x;l)$ exists for every $l$, is continuous as a function of the direction $l$, is bounded by $M$:
$$
\max_{l\in \Sn} |f'(x;l)|\leq M \;,
$$
and the restriction
$f_l(\cdot) = f'(x;l+\Prj\transp_{n-1,l}(\cdot)):\R^{n-1}\to \R$ of its directional derivative to $l+\myspan \{l\}^\perp$ is also $M$-directed subdifferentiable at $0_{n-1}$ for all $l\in \Sn$.

We say that a function is {\em directed subdifferentiable}, iff it is $M$-directed subdifferentiable for some $M\geq 0$.
\end{definition}

Observe that the restrictions $f_l$ are constructed in the same way as for convex and DC functions above (also see Fig.~\ref{fig:Example1SecondMethod} for an illustration).

\begin{remark}Even though the above definition might look cumbersome, the directed subdifferentiable functions are very common in applications. All Lipschitz functions definable on $o$-minimal structure (which include semi-algebraic functions \cite{DaniilidisPang,Iof2008}) are directed subdifferentiable. Even though this fact is fairly obvious for anybody sufficiently familiar with definable functions, we provide a sketch of the proof for convenience in Lemma~\ref{lem:definable}. We would like to thank Aris Daniilidis and Dmitry Drusviatskiy for providing us with helpful hints for the proof.
\end{remark}

\begin{remark} We would like to mention that the study of definable functions and their generalized subdifferentials has recently drawn much attention (e.g.\ see \ch{\cite{DaniilidisPang,BolteDanLewTame,DrusIoffeLewis})}. The class of definable functions is very broad, and  vast majority of functions encountered in applications belongs to this class. This restriction helps focussing on results crucial for optimization theory \cite{Iof2008} without being distracted by abnormal behaviour of exotic examples.
\end{remark}

We next show that locally Lipschitz definable functions are directed differentiable. We first provide the definition of functions defined on $o$-minimal structure and that of tame functions. We follow \cite{BolteDanLewTame}.

\begin{definition}[o-minimal structure] 
%%% An o-minimal structure on
%%% $(\R,+, .)$ is a sequence of Boolean algebras ${\cal O} = \{{\cal O}_n\}$ of ``definable'' subsets of $\R^n$, such that for each $n\in \N$
\sloppy{A sequence of Boolean algebras ${{\cal O} = \{{\cal O}_n\}}$ of \emph{definable} subsets of $\R^n$
is called an \emph{o-minimal structure} on $(\R,+, .)$, iff 
for each $n\in \N$}
\begin{itemize}
\item[(i)] if $A \in {\cal O}_n$, then $A\times \R$, $\R\times A \in {\cal O}_{n+1}$;
\item[(ii)] if $\Pi:\R^{n+1}\to \R^n$ is the canonical projection onto $\R^n$ then for any $A \in {\cal O}_{n+1}$, the set $\Pi(A)$ fulfills $\Pi(A) \in {\cal O}_n$;
\item[(iii)] 
             every set of the form $\{x\in \R^n\,:\, p(x) = 0\}$, where $p:\R^n\to \R$ is a polynomial function, belongs to ${\cal O}_{n+1}$, i.e.~${\cal O}_n$ contains the family of algebraic 
             subsets of $\R^n$; 
\item[(iv)] ${\cal O}_1$ consists exactly of finite unions of intervals and points.
\end{itemize}

A function $f:S\subset \R^n\to \R^m$ is said to be \emph{definable} in $\cal O$, iff its graph (as a subset of $\R^n\times \R^m$) is definable in $\cal O$.
\end{definition}

\begin{definition} A subset $A \subset \R^n$ is called {\em tame}, iff for every $r > 0$ there exists an o-minimal structure ${\cal O}$ over $\R$ such that the intersection $A \cap [-r,r]^n$ is definable in $\cal O$. Similarly a function $f: U \subset \R^n \to \R^m$ is called \emph{tame}, iff its graph (as a subset of $\R^n\times \R^m)$ is tame.
\end{definition}

One of the key results in the theory of definable functions is the following monotonicity theorem (see \cite[Chap.~5]{BOCHN/COS/ROY:1987} and \cite[Chap.~2]{Coste}).

\begin{theorem}[Monotonicity theorem]\label{thm:monot} 
Let $f : \ch{]a, b[} \to \R$ be a definable function. There exists a finite subdivision $a = a_0 < a_1 < \dots < a_k = b$ such that, on each interval $\ch{]a_i, a_{i+1}[}$, $f$ is continuous and either constant or strictly
monotone.
\end{theorem}

The fact that locally Lipschitz definable functions are directed subdifferentiable follows from the next lemma. One needs to apply it recursively $n$ times to the directional derivatives of directional derivatives to obtain the desired result. We only provide a sketch of the proof, as it is not the core result of the paper, and, as we mentioned before, the proof is evident for anybody familiar with o-minimal structures \cite{Coste,Iof2008}. To generalize the result to tame functions, observe that a Lipschitz tame function is definable in a neighborhood of each point, hence, the proof carries over to tame functions.

\begin{lemma}\label{lem:definable} Let $f:\R^n\to \R$ be locally Lipschitz around $x\in \R^n$ and definable on some o-minimal structure. Then the directional derivative $f'(x;\cdot)$ of $f$ at $x$ exists, is globally Lipschitz and definable on the whole space $\R^n$.
\end{lemma}
\begin{proof}
\sloppy{The existence of the directional derivative follows from the monotonicity theorem. It is a trivial exercise to show that the function ${g_{x}(y,t):\R^n\times \ch{]0,\infty[}\to \R}$,}
$$
g_{x}(y,t) \chR{:=} \frac{f(x+ty)-f(x)}{t}
$$
is definable (it is a composition of definable functions). The monotonicity theorem implies that for every $y\in \R^n$ there exists $r_1>0$ such that $g_x$ is monotone on $y\times \ch{]0,r_1[}$. Moreover, since $f$ is locally Lipschitz around $x$, there is an $r_2>0$ such that $g$ is bounded on $\{y\}\times \ch{]0, r_2[}$. Hence the limit
$$
f'(x;y) = \lim_{t\downarrow 0} g_{x}(y,t)
$$
exists for every $y\in \R^n$.

The Lipschitzness of the directional derivative follows from the local Lipschitzness of the function $f$ around $x$.

%%% !!! fill word "now" to avoid big horizontal gaps
\sloppy{It remains to show \chR{now} that the directional derivative $h_x:\R^n\to \R$, ${h_x(y) \chR{ := f'}(x;y)}$ is definable. We use the argument from \cite{Iof2008}. For any definable $\psi:\R^n\times A \to \R$, where $A\in \R^p$ is a definable set, the inf function}
$$
\varphi(x) \chR{:=} \inf_{s\in A} \psi(x,s)
$$
is definable (this result is not difficult to prove using the standard technique described, e.g.~in \cite[proof of Proposition~1.2]{Coste}. Observe that
\begin{align*}
h_x(y) = \chR{f'}(x;y)
& = \lim_{t\downarrow 0} g_{x}(y,t) \\
& = \sup_{\varepsilon \in \ch{]0,r_1[}}\inf_{t\in {]0,\varepsilon[}}g_x(y,t)\\
&  = \sup_{\varepsilon \in {]0,r_1[}}\inf_{t\in {]0,+\infty[}}g_x(y,\min\{t,\varepsilon\})\\
&  = - \inf_{\varepsilon \in {]0,r_1[}}[-\inf_{t\in {]0,+\infty[}}g_x(y,\min\{t,\varepsilon\})].
\end{align*}
The function $g_x(y,\min\{t,\varepsilon\}):\R^n\times \ch{]0,+\infty[}\times \ch{]0,+\infty[}$ is definable, hence, $h_x(y)$ is definable as an inf\chR{-}function.
\end{proof}

We next show that arithmetic operations on directed subdifferentiable functions as well as pointwise min and max operations again yield directed subdifferentiable functions.

\begin{lemma}\label{lem:MCalc} Let $f_i:\R^n\to \R$ be $M_i$-directed subdifferentiable at $x$ for $i=1,2$ and $\alpha,\beta \in \R$.
Then
\begin{itemize}
\item[(i)]{ $\alpha f_1 + \beta f_2$ is $|\alpha| M_1+ |\beta| M_2$-directed subdifferentiable at $x$,}
\item[(ii)]{ $f_1\cdot f_2$ is $M$-directed subdifferentiable at $x$ with
$$
M \chR{:=} M_1|f_2(x)|+M_2|f_1(x)| \;.
$$
           }
\item[(iii)]{ $f_1/f_2$ is $M$-directed subdifferentiable at $x$ with
$$
M \chR{:=}  \frac{M_1|f_2(x)|+M_2 |f_1(x)|}{|f_2(x)|^2} \;,
$$
              if $f_2(x)\neq 0$.
            }
\end{itemize}
\end{lemma}
\begin{proof}
Let $f_1,f_2:\R^n \to \R$ be $M_1$ and $M_2$-directed subdifferentiable respectively, $x\in \R^n$ and let $f \chR{:=} \alpha f_1+ \beta f_2$. Then from Remark~\ref{rem:DirDer}
\begin{equation}\label{eq:ddSum01}
f'(x;l) = \alpha f'_1(x;l)+ \beta f'_2(x;l) \quad \forall l\in \R^n \;,
\end{equation}
and therefore
 \begin{align*}
    \max_{l\in \Sn}\left|f'(x;l) \right| & \leq |\alpha| \max_{l\in \Sn}\left|f'_1(x;l)\right|+ |\beta| \max_{l\in \Sn}\left|f'_2(x;l) \right| \\
    & \leq M := |\alpha| M_1+ |\beta| M_2 \;.
 \end{align*}
Moreover, \eqref{eq:ddSum01} yields
$$
f_l(\cdot) = \alpha (f_1)_l(\cdot)+ \beta (f_2)_l(\cdot) \;.
$$
Since both $(f_1)_l(\cdot)$ and $(f_2)_l(\cdot)$ are directionally differentiable, by the same argument as above we have
$$
\max_{y\in \Snminus}\left|f_l'(0;y)\right| \leq M = |\alpha| M_1+ |\beta| M_2\quad \forall l\in \Sn \;.
$$
Repeating the same argument over and over again until we reach the relevant one-dimensional restrictions, we prove that $f$ is $M$-directed subdifferentiable.

The proof for the rest of the cases is analogous, except for several details that we highlight below.

Observe that for the product $f \chR{:=}f_1\cdot f_2$ of two directed subdifferentiable functions we have from Remark~\ref{rem:DirDer}
\begin{align*}
\sup_{l\in \Sn}\left|f'(x;l)\right|
 & \leq |f_1(x)|\sup_{l\in \Sn}\left|f'_2(x;l)\right| +|f_2(x)|\sup_{l\in \Sn}\left|f'_1(x;l)\right|\\
 & \leq |f_1(x)|M_1+|f_2(x)|M_2 \;;
\end{align*}
as we go one level down we get for each $l\in \Sn$
\begin{align*}
\sup_{y\in \Snminus}\left|f'_l(0;y)\right|
 & = \sup_{y\in \Snminus}\left|f_1(x)(f_1)'_l(0;y)+ f_1(x)(f_2)'_l(0;y)\right|\\
 & \leq |f_1(x)|M_1+|f_2(x)|M_2 \;;
\end{align*}
and so on until dimension~1. The derivation is analogous for the ratio.
\end{proof}

It can be shown with Remark~\ref{rem:DirDerMaxMin} that the pointwise maximum and minimum functions are directed subdifferentiable.
\begin{lemma}\label{lem:MCalcMaxMin} Let $f_i:\R^n\to \R$ be $M_i$-directed subdifferentiable at $x$ for $i\in I$, where $I$
is a finite index set.
Then the pointwise maximum function
$$
f_{\max}(x) \chR{:=} \max_{i\in 1:p}f_i(x)
$$
is $M_{\max}$-directed subdifferentiable at $x$ with
$$
M_{\max} \chR{:=} \max_{i\in I(x)}|M_i|, \text{ with }  I(x) \chR{:=} \{i\in I\,|\, f_i(x) = f_{\max}(x)\} \;,
$$
and likewise the pointwise minimum function
$$
f_{\min}(x) \chR{:=} \min_{i\in 1:p}f_i(x)
$$
is $M_{\min}$-directed subdifferentiable with
$$
M_{\min} \chR{:=} \max_{i\in J(x)}|M_i|, \text{ with }  J(x) \chR{:=} \{i\in J\,|\, f_i(x) = f_{\min}(x)\} \;.
$$
\end{lemma}
We omit the proof of this Lemma, as it is analogous to the proof of Lemma~\ref{lem:MCalc}, using the previous Remark~\ref{rem:DirDerMaxMin}.

\subsection{Extension to Directed Subdifferentiable Functions}

The directed subdifferential is now extended to the class of
directed subdifferentiable functions in which no DC structure for the
function or the directional derivative is present any longer. Instead, we will
base the construction on directional derivatives only.

\begin{definition}\label{def:ddd}{\bf (Second definition of the directed subdifferential via directional derivatives)} Let $f:\R^n\to \R$ be directed subdifferentiable, then for any fixed $x\in \R^n$ the directed subdifferential of $f$ at $x$ can be defined as follows:
\begin{itemize}
  \item[(i)] For $n=1$ let $\dd f(x) \ch{:=} \overrightarrow{[-f'(x;-1);f'(x;1)]}$;
  \item[(ii)] For $n\geq 2$ let
  $$
   \dd f(x) \ch{:=} \left(\dd f_l, f'(x;l)\right)_{l\in S_{n-1}} \;,
  $$
  where $f_l:\R^{n-1}\to \R$ is defined by $f_l(\cdot) \ch{:=} f'(x;l+\Prj\transp_{n-1,l}(\cdot))$.
\end{itemize}
\end{definition}

The next statement shows that the directed subdifferential is a well-defined
directed set for directed subdifferentiable functions.

\begin{lemma}\label{lem:DefCorrect} Let  $x\in \R^n$ and assume $f:\R^n\to \R$ is directed subdifferentiable. Then the object described by Definition~\ref{def:ddd} is a directed set. Besides, if $f$ is $M$-directed subdifferentiable, then
\begin{equation}\label{eq:DefCorrect001}
\|\dd f(x)\| = \inf \{M\,|\, f \text{ is } M\text{-directed subdifferentiable at } x\} \;,
\end{equation}
moreover, the infimum is attained.
\end{lemma}
\begin{proof} For $n=1$ the statement is trivial: indeed, given $f:\R\to \R$ such that $f$ is $M$-directed subdifferentiable at $x\in \R$, we have
$$
\dd f(x) = \overrightarrow{[-f'(x;-1);f'(x;1)]} \;.
$$
Clearly, $\|\dd f(x)\| = \max\{|f'(x;-1)|, |f'(x,1)|\}\leq M$.

Assume now that the statement is true for $n=k-1$. We will prove it for $n=k$. Consider $f:\R^k\to \R$ such that $f$ is $M$-directed subdifferentiable at $x\in \R^k$.
By Definition~\ref{def:ddd} of the directed subdifferential
$$
\dd f(x) = \left(\dd f_l, f'(x;l)\right)_{l\in \Snvar{k}} \;.
$$

Observe that for every $l\in \Snvar{k}$ the function $f_l$ is $M$-directed subdifferentiable, and hence by the induction assumption
$\dd f_l$ is a directed set, moreover,
\begin{equation}\label{eq:DefCorrect002}
\|\dd f_l \|\leq M \qquad \forall l\in \Snvar{k} \;,
\end{equation}
and hence the first component of $\dd f_l$ in \eqref{eq:DefCorrect001} is a bounded map from $\Snvar{k}$ to $\mathcal{D}(\R^{k-1})$. Further, since $f$ is $M$-directed subdifferentiable,
its directional derivative $f'(x;\cdot)$ is continuous and bounded by $M$ on the unit sphere:
\begin{equation}\label{eq:DefCorrect003}
\sup_{l\in \Snvar{k}} |f'(x;l)| \leq M \;.
\end{equation}
From \eqref{eq:DefCorrect002} and \eqref{eq:DefCorrect003} we have $\|\dd f(x)\|\leq M$.

It is self-evident from the definition of an $M$-directed subdifferentiable function that the infimum in \eqref{eq:DefCorrect001} is attained.
\end{proof}

We have already shown in Corollary~\ref{cor:dir_subd_qd} that Definition~\ref{def:ddd} extends the directed subdifferential for  quasidifferentiable functions introduced in \cite{BaiFarRosh2012a} and restated in Definition~\ref{Def_Dir_subdif_QD}.

We would like to point out that the difficulty to construct the `max+min' representation for the directional derivative has been repelling people from applying quasidifferentials to solve nonsmooth optimization problems (apart from the problem of non-uniqueness and `inflation' after applying calculus rules already addressed in \cite{BaiFarRosh2012b}). On the other hand, the excellent calculus rules of the quasidifferential, e.g.~the exact sum rule, are often not taken into account. Therefore, the QD representation can be gained by applying the exact calculus rules. Assuming that the function consists of (nonsmooth) outer operations like `max or min' and of (differentiable) simple inner functions, one can iterate the QD representation formulas from the inner functions to end up with the QD representation of the (complicated) outer functions. This idea (similar to the ideas in algorithmic/automatic differentiation) can also be applied for the directed subdifferential.
While the directed subdifferential is essentially based on the same idea as quasidifferential, our approach solves the fundamental problem of constructing the DC representation of the directional derivative --- we simply do not need it, as we only utilize directional derivatives in our construction.

\subsection{The Class of Directed Subdifferentiable Functions}\label{SecClassOfDirSubFunc}

As we have mentioned before, the class of directed subdifferentiable functions is rather large and includes many important and well-studied classes of functions. For the convenience of the reader, we remind the definitions of (strongly) amenable and lower- and upper-$\Cfvar{k}$  functions first.

The following definition of amenable functions is taken from
\cite[Definition~10.23]{RockWets1998}
and \cite[Sec.~3.2, remarks before Corollary~3.76]{MordukBookI2006}.

\begin{definition}
\label{Def_amenable}
   The function $f: \R^n \rightarrow \R \cup \{ \pm \infty \}$
   is called \emph{amenable}, iff it has the form
   $f = g \circ \varphi$ with $g: \R^m \to \R \cup \{ \pm \infty \}$ proper, l.s.c., convex
   and $\varphi \in C^1(\R^n; \R^m)$ such that the following constraint qualification
   holds. Whenever we have the equation
    \begin{align*}
       \JacMat{\varphi}{x}^\trans y = 0 \quad\text{for some $y \in N_D(\varphi(x))$} \,,
    \end{align*}
   in a neighborhood of $x$, then this vector $y$ must be zero. \\
   Here,
   $\JacMat{\varphi}{x}$ denotes the Jacobian of $\varphi$ at $x$ and
   $N_D(z)$ is the normal cone at $z = \varphi(x)$
   (see \cite[Definition~6.3]{RockWets1998}) to the closure of the effective domain of $g$,
   i.e.~$D \chR{:=} \cl({\rm dom\,} g)$.
   An amenable function $f: \R^n \rightarrow \R$ is called \emph{strongly
   amenable}, iff $\varphi \in C^2$.
\end{definition}

We cite the definition lower-$C^k$ functions from \cite[(1.6)]{Rock1982} and \cite[Definition~10.29]{RockWets1998}; upper-$C^k$ functions can be defined in a symmetric way.

\begin{definition}
\label{Def_lower_Ck}
   The function $f: \R^n \rightarrow \R $
   is called \emph{lower-$C^k$ with $k \in \N \cup \{\infty\}$},
   \chR{iff} it has the form
    \begin{align*}
       f(x) & = \sup_{p \in P} F(x, p) \quad(x \in \R^n) \,,
    \end{align*}
   where $P$ is a compact topological space and $F: \R^n \times P \to \R$
   is a function that has partial derivatives w.r.t.~$x$ up to order $k$
   with both $F$ and these derivatives (jointly) continuous w.r.t.~$(x,p)$.
\end{definition}

We state in the following remark a list of function classes that are
directed subdifferentiable.

\begin{remark}\label{rem:dir_subdiff}
   The following types of functions $f: \R^n \to \R$ are directed
   subdifferentiable:
    \begin{itemize}
      \item[(i)] convex functions
      \item[(ii)] DC functions
      \item[(iii)] quasidifferentiable functions
      \item[(iv)] amenable (and strongly amenable) functions
      \item[(v)] lower-$\Cfvar{k}$ and upper-$\Cfvar{k}$ functions for $k \in \N$
    \end{itemize}
   It is well-known that DC functions are quasidifferentiable and that lower-$\Cfvar{k}$ 
   and upper-$\Cfvar{k}$ functions are DC for $k \geq 2$ (see \cite[Theorem~6]{Rock1982}).  
   The first three classes belong to the class of directed subdifferentiable
   functions by~\cite[Lemma~3.4]{BaiFarRosh2012a} which states that
   a directional derivative of a quasidifferentiable function is also
   quasidifferentiable.

   (iv) follows from~\cite[Corollary~5.8]{BaiFarRosh2012a}, while
   (v) follows from~\cite[Corollary~5.4]{BaiFarRosh2012a} for $k \geq 2$
   and from~\cite[Proposition~5.5]{BaiFarRosh2012a} for $k = 1$.
   In both cases it was shown that the function classes (iv)--(v), see the references
   in~\cite{BaiFarRosh2012a} with the definitions, belong to the bigger class of quasidifferentiable functions for all $k \in \N$.
\end{remark}

Curiously enough, it appears that the class of directed subdifferentiable functions is substantially wider than QD. We next give an example of a directed subdifferentiable function $f:\R^2\to \R$ defined as
$$
   f(x,y) \chR{:=} \inf_{k\in 1: \infty} \left|x-\frac{y}{k}\right|
$$
which is not quasidifferentiable. This particular example is motivated by an idea of J.~C.~H.~Pang and based on the well-known one-dimensional construction discussed in detail in~\cite[Example~3.5]{BaiFarRosh2012a} (see also the references therein).

The idea presented in this example can be seen as a prototype to define
a family of functions that are not quasidifferentiable but directed
subdifferentiable. In this sense it is a universal construction to create
further examples.

To show that this function is Lipschitz, we need a small technical result
which also states the Lipschitz constant. The proof is self-evident, hence, omitted; for a more general result with upper semi-continuous parameters see \cite[Sec.~2.8]{Cla1983}.

\begin{lemma}\label{lem:MinMaxLip} Let $f_i:\R^n\to \R$, $i=1,2$, be globally Lipschitz with a common constant $L$. Then their pointwise maximum and minimum are also Lipschitz with the same constant.
\end{lemma}

\begin{exam}\label{ex:NonQD} Let $f:\R^2\to \R$ be defined as follows:
$$
f(x,y) \chR{:=} \inf_{k\in 1: \infty} \left|x-\frac{y}{k}\right|
$$

(i) Observe that for every $k\in N$ the linear function $x-y/k$ is  globally Lipschitz with the constant $L_k \chR{:=} \sqrt{1+1/k^2}\leq \sqrt{2}=:L$.

From Lemma~\ref{lem:MinMaxLip} we deduce that for any $N\in \N$ the function
$$
f_N(x,y) \chR{:=} \min_{k\in 1:N}\left|x-\frac{y}{k}\right|= \min_{k\in 1:N}\max\left\{x-\frac{y}{k}, -x+\frac{y}{k}\right\}
$$
is Lipschitz with constant $L$.

Pick an arbitrary pair of points $x_1, x_2\in \R^2$.
It is not difficult to observe that
$$
f(x,y) = \lim_{N\to \infty} f_N(x,y) \;,
$$
therefore, for every $\varepsilon > 0 $ there exists $M\in \N$ such that
$$
|f(x_1) - f_M(x_1)| < \frac{\varepsilon}{2}\|x_1-x_2\|, \quad |f(x_2) - f_M(x_2)| < \frac{\varepsilon}{2}\|x_1-x_2\| \;.
$$
Hence,
\begin{align*}
|f(x_1) - f(x_2)| & \leq \varepsilon \|x_2-x_1\|+|f_M(x_1) - f_M(x_2)|\leq (L+\varepsilon)\|x_2 - x_1\| \;, \\
|f(x_1) - f(x_2)| & \leq \lim_{\varepsilon \downarrow 0 }(L+\varepsilon)\|x_2 - x_1\| =  L\|x_2 - x_1\| \;,
\end{align*}
and by the arbitrariness of $x_1$ and $x_2$ the function $f$ is globally Lipschitz with constant $L = \sqrt{2}$.

(ii) The function $f$ is positively homogeneous, hence, it is Dini directionally differentiable at zero with
$$
f'(0;l) = f(l)\qquad \forall\, l\in \R^2 \;.
$$
It follows that the function is directionally differentiable in 0 in the sense of \cite[Sec.~I.3]{DemRub1995}
and from the Lipschitzness of $f$ that it is also Hadamard directionally differentiable.

(iii) We next show that even though this function is both Lipschitz and directionally differentiable, it is not quasidifferentiable. Assume the contrary: then there exist two sublinear functions $g,h:\R^2 \to \R$
such that
$$
f'(0;(x,y)) = f(x,y) = g(x,y) - h(x,y)\qquad \forall (x,y) \in \R^2 \;.
$$
Consider the restriction of this identity to the line $\R\times \{1\}$. Then
$$
f(\cdot,1) = \inf_{k\in 1:\infty} \left|\cdot-\frac{1}{k}\right| = g(\cdot, 1) - h(\cdot, 1) \;,
$$
where $g(\cdot, 1)$ and $h(\cdot, 1)$ are convex functions. It is well-known (see the discussion in~\cite{BaiFarRosh2012a} and references therein) that the function $f(\cdot,1)$ is not of bounded variation, hence, can not be represented as a DC function. This contradicts our assumption that $f$ is quasidifferentiable.

(iv) We next show that the function $f$ is $M$-directed subdifferentiable at the origin with $M = \sqrt{2}$. To do that, we state that the directional derivative $\dirDeriv{f}{0}{l}$ of $f(\cdot)$ exists for every $l \in \Stwo$ by~(ii). We need to show that
 \begin{itemize}
   \item[1)] the directional derivative $\dirDeriv{f_l}{0}{\eta}$ of the auxiliary
             function $f_l(\cdot)$ exists for every $\eta = \pm 1$
   \item[2)] the directional derivatives $\dirDeriv{f}{0}{l}$ and
             $\dirDeriv{f_l}{0}{\eta}$ are bounded by $M$ for every $l \in \Stwo$
             and $\eta = \pm 1$, i.e.
              \begin{alignat*}{2}
                 |\dirDeriv{f}{0}{l}| & \leq M \quad(l \in \Stwo) \;, & \quad
                 \max\{|f_l'(0;1)|,|f_l'(0;-1)|\} & \leq M \;.
              \end{alignat*}
 \end{itemize}
We first have a closer look at the function, considering its values on different areas in $\R^2$.
It is not difficult to see that
$$
f(x,y)
= \left\{
\begin{array}{rrl}
|x| \;,                 & xy\leq 0 & \text{(area A)}\\
|x|-|y| \;,             & \frac{x}{y}\geq 1 \;,\ y\neq 0 & \text{(area B)}\\
\frac{|y|}{k}-|x|\;,& \frac{x}{y}\in
     \left[\frac{1}{2}\left(\frac{1}{k}+\frac{1}{k+1}\right), \frac{1}{k}\right] \;, \ k\in \N
                     & \text{(area $C_k$)}\\
|x|- \frac{|y|}{k+1} \;,   & \frac{x}{y}\in
     \left[\frac{1}{k+1}, \frac{1}{2}\left(\frac{1}{k}+\frac{1}{k+1}\right)\right] \;, \ k\in \N
                     & \text{(area $D_k$)}.
\end{array}
\right.
$$
Since $\dirDeriv{f}{0}{l} = f(l)$ by~(iii), the absolute value of the directional
derivative is bounded by $M$.

The areas $A$ and $B$ and the areas $C_k$ and $D_k$ for several initial values of $k$ are shown on  \chR{Fig.}~\ref{fig:AreasForExample}.

\begin{figure}[h]
\centering \includegraphics[width=200pt, keepaspectratio ]{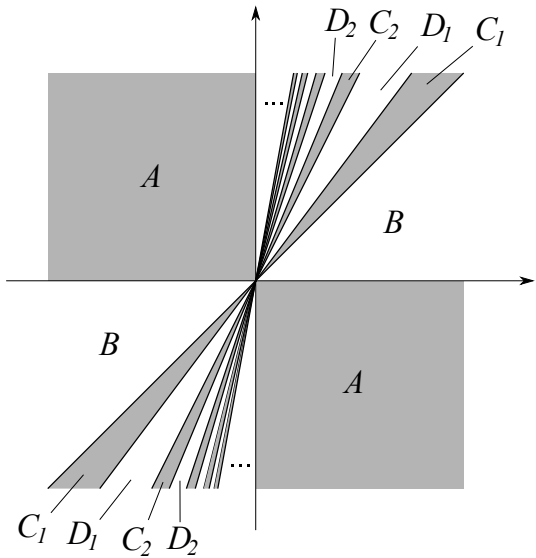}
\caption{Regions on which the function $f$ from Example~\ref{ex:NonQD} is linear} \label{fig:AreasForExample}
\end{figure}
Observe that the function $f$ is piecewise linear in a neighbourhood of every point except for the line $H = \{0\}\times \R$. Therefore, its directional derivative exists and is bounded by the Lipschitz constant of $f$ everywhere except for $H$, and hence the restriction of $f$ to the relevant affine subspaces is also defined and is bounded by the Lipschitz constant everywhere on $\R^2\setminus H$. It remains to show that the restriction $f_l$ of $f$ is directionally differentiable for $l\in \Stwo\cap H$. Observe that $\Stwo\cap H = \{(0,1), (0,-1)\}$. For $l^+ = (0,1)$ and
$l^- = (0,-1)$ we have
$$
f_{l^+}(t) = f(t,1) = \inf_{k\in \N}\left|t - \frac{1}{k}\right|, \quad
f_{l^-}(t) = f(t,-1) = \inf_{k\in \N}\left|t + \frac{1}{k}\right| = f_{l^+}(-t) \;.
$$
It was shown in \cite[Example 3.5]{BaiFarRosh2012a} that the function $f_{l^+}$ is directionally differentiable with $\max_{v\in \{-1,1\}}|f'_{l^+}(0;v)|\leq 1<L$. We have hence shown that the function $f$ is $M$-directed subdifferentiable.
\end{exam}

\begin{remark} Note that the function $f$ in Example~\ref{ex:NonQD} is neither definable nor even tame. To see this, observe that if $f$ would be definable, so would be its restriction $f_{l^+}$ (it can be represented as an orthogonal projection of the intersection of the graph of $f$ with the set $\{y=1\}$). Every definable function must satisfy the monotonicity theorem, in particular, there must exist an $\varepsilon>0$ such that $f_{l^+}$ is strictly monotone on $(0,\varepsilon)$. It is not the case, as $f_{l^+}$ is oscillating in the neighborhood of $x$. This also shows that the function is not tame, as this behavior happens in any bounded box around the origin.
\end{remark}

\section{Conclusions} As seen in Sections~\ref{SecDefDirSubFunc} and \ref{SecClassOfDirSubFunc}, the class of directed
subdifferentiable functions is rather large and is closed under
the arithmetic and boolean operations. The definition is motivated
by the directed subdifferential and includes (but is not limited to) such important subclasses as DC, quasidifferentiable, amenable and
Lipschitz definable functions.

The next steps will be the extension of calculus rules of the directed
subdifferential for directed subdifferentiable functions. We indeed hope
that the exact calculus rules for the directed subdifferential for QD
functions carries over to this bigger class. In this respect we expect
an exact sum rule for two directed subdifferentiable functions to hold without any additional regularity assumptions. Since the calculus relies heavily on directional derivatives in a recursive way, we anticipate easy and exact calculus rules for the directed subdifferential of maximum and minimum of directed subdifferentiable functions.  Based on these calculus rules we expect to elaborate formulas for the directed subdifferential of amenable and
lower-$C^k$ functions as well as optimality conditions, the mean-value
theorem and chain rules in an equality form.

\section*{Acknowledgements}
The research is partially supported by The Hermann Minkowski Center for Geometry at Tel Aviv University, Tel Aviv, Israel and by the University of Ballarat `Self-sustaining Regions Research and Innovation Initiative', an Australian Government Collaborative Research Network (CRN).

We would like to acknowledge the discussions with Jeffrey C.~H.~Pang starting in October 2009 about Lipschitz definable functions. During the communications he suggested a construction to illustrate the existence of Lipschitz functions which are not quasidifferentiable. This motivated us to consider Example~\ref{ex:NonQD} which is based on a well-known one-dimensional example of a function that is quasidifferentiable but not a DC function (see~\cite[Example~3.5]{BaiFarRosh2012a}).
The authors are also grateful to Aris Daniilidis and Dmitriy Drusvyatskiy for a helpful discussion on Lipschitz definable/semi-algebraic functions.

\def\cprime{$'$} \def\cydot{\leavevmode\raise.4ex\hbox{.}}
  \def\cfac#1{\ifmmode\setbox7\hbox{$\accent"5E#1$}\else
  \setbox7\hbox{\accent"5E#1}\penalty 10000\relax\fi\raise 1\ht7
  \hbox{\lower1.15ex\hbox to 1\wd7{\hss\accent"13\hss}}\penalty 10000
  \hskip-1\wd7\penalty 10000\box7} \def\Dbar{\leavevmode\lower.6ex\hbox to
  0pt{\hskip-.23ex \accent"16\hss}D} \def\dbar{\leavevmode\hbox to
  0pt{\hskip.2ex \accent"16\hss}d}

\end{document}